\documentclass[onefignum,onetabnum]{siamart171218}

\usepackage{lipsum}
\usepackage{amsfonts}
\usepackage{graphicx}
\usepackage{epstopdf}
\usepackage{algorithmic}
\usepackage{physics}
\usepackage{amsmath}
\ifpdf
  \DeclareGraphicsExtensions{.eps,.pdf,.png,.jpg}
\else
  \DeclareGraphicsExtensions{.eps}
\fi

\definecolor{steelblue}{HTML}{A1BDC7}
\definecolor{orange}{HTML}{D98C21}
\definecolor{silver}{HTML}{B0ABA8}
\definecolor{rust}{HTML}{B8420F}
\definecolor{seagreen}{HTML}{2E6B69}
\definecolor{joshua}{HTML}{FBDC7F}
\definecolor{darksky}{HTML}{154c79}

\colorlet{lightsilver}{silver!30!white}
\colorlet{darkorange}{orange!85!black}
\colorlet{darksilver}{silver!85!black}
\colorlet{darksteelblue}{steelblue!85!black}
\colorlet{darkrust}{rust!85!black}
\colorlet{darkseagreen}{seagreen!85!black}

\hypersetup{colorlinks=true,linkcolor=darkrust,citecolor=darkseagreen,urlcolor=darksilver}

\newsiamremark{remark}{Remark}
\newsiamremark{hypothesis}{Hypothesis}
\crefname{hypothesis}{Hypothesis}{Hypotheses}
\newsiamthm{claim}{Claim}

\headers{Optimal Krylov on Average}{Qi Luo and Florian Sch\"afer}

\title{Optimal Krylov on Average
}

\author{Qi Luo\thanks{Georgia Institute of Technology 
  (\email{qluo48@gatech.edu}).}
\and Florian Schäfer \thanks{Georgia Institute of Technology 
  (\email{fts@gatech.edu}).}
}

\usepackage{amsopn}

\makeatletter
\newcommand*{\addFileDependency}[1]{
  \typeout{(#1)}
  \@addtofilelist{#1}
  \IfFileExists{#1}{}{\typeout{No file #1.}}
}
\makeatother

\newcommand*{\myexternaldocument}[1]{%
    \externaldocument{#1}%
    \addFileDependency{#1.tex}%
    \addFileDependency{#1.aux}%
}

\ifpdf
\hypersetup{
  pdftitle={Optimal Krylov On Average},
  pdfauthor={Qi Luo and Florian Schäfer}
}
\fi

\myexternaldocument{ex_supplement}

\begin{document}

\maketitle

\begin{abstract}
  We propose an adaptive randomized truncation estimator for Krylov subspace methods that optimizes the trade-off between the solution variance and the computational cost, while remaining unbiased. The estimator solves a constrained optimization problem to compute the truncation probabilities on the fly, with minimal computational overhead. The problem has a closed-form solution when the improvement of the deterministic algorithm satisfies a diminishing returns property. We prove that obtaining the optimal adaptive truncation distribution is impossible in the general case. Without the diminishing return condition, our estimator provides a suboptimal but still unbiased solution. We present experimental results in GP hyperparameter training and competitive physics-informed neural networks problem to demonstrate the effectiveness of our approach.

\end{abstract}


\section{Introduction}

\subsection*{The problem} 
Linear systems frequently appear in the inner loop of dynamical systems arising from the Gauss-Newton method \cite{gauss1877theoria}, competitive gradient descent \cite{schafer2019competitive}, and implicit methods for differential equations \cite{baraff2001physically}. Krylov subspace methods, especially the conjugate gradient algorithm (CG) and minimal residual algorithms, are popular for solving these problems. Achieving floating point accuracy for every linear system solve is too computationally expensive. This necessitates the development of early truncation techniques for these iterative methods. Deterministic early truncation, while reducing computational cost, results in systematic errors in the final output, motivating the search for alternative truncation methods.

\subsection*{Randomized truncation estimators\nopunct} eliminate truncation bias in exchange for variance. Two popular approaches include the Single Sample estimator \cite{10.1214/15-STS523} and the Russian Roulette estimator \cite{P-766}. The latter is especially suitable for iterative methods, because each iteration relies on the information from all previous ones. 
The Russian Roulette estimator truncates the algorithm at random and adjusts the weight of each surviving term by the reciprocal of the survival probability to ensure an unbiased solution. 
Since the estimator remains unbiased regardless of the chosen probability distribution, the focus shifts to selecting a truncation distribution that optimizes the trade-off between variance and computational cost. 

\subsection*{Existing methods}
Recent work aims to improve the sampling technique, particularly for the conjugate gradient algorithm (CG). The Unbiased LInear System SolvEr (ULISSE) \cite{filippone2015enabling} suggests sampling the iteration term from an exponential distribution. 
RR-CG\cite{potapczynski2021bias} justifies this choice, noting that it matches the error decay of CG and maximizes the relative optimization efficiency (ROE) metric \cite{beatson2019efficient}. 
RR-CG also tunes the parameters to improve the trade-off between variance and computational cost. 
However, these existing approaches determine the sampling distribution \emph{a-priori}, in stark contrast to the instance-specific optimality of Krylov subspace methods.
    


\subsection*{Our approach: Adaptively Subsampled Conjugate Gradient (AS-CG)}
This work proposes the Adaptively Subsampled estimator (AS estimator), a truncation estimator that preserves an optimal trade-off between variance and computational cost in the stochastic setting. We call its application to the conjugate gradient algorithm AS-CG. In CG, each iteration minimizes the energy norm of the error, with search directions being A-orthogonal. Using this Krylov information, we construct a constrained optimization problem for the expected squared energy norm of the error with respect to the average number of iterations. This problem admits a closed-form solution when the deterministic CG's improvement decreases with increasing iteration number. 
We prove that without this diminishing return condition, obtaining the optimal on-the-fly truncation distribution is impossible. Our estimator still provides an unbiased, albeit suboptimal, solution in this case. 
We also provide a modification of our approach that applies to minimal residual methods.

\subsection*{Summary of contributions}
We introduce the AS estimator, a novel randomized truncation estimator, to overcome the lack of adaptivity and theoretical optimality of existing truncation estimators for CG and minimal residual iterative methods.
We prove that under a diminishing returns assumption, AS-CG achieves the optimal cost-variance trade-off.
We apply our AS estimator to GP hyperparameter optimization and Competitive Physics Informed Neural Network training to demonstrate its effectiveness.
Compared with the existing RR-CG estimator, AS-CG shows a better cost-variance trade-off and improved stability. 

\section{Our Method}

\subsection{Suitable Krylov subspace methods}
We provide background on the conjugate gradient (CG) and two minimal residual methods, the conjugate residual (CR) and the general minimal residual (GMRES) methods, which are suitable for our approach. We denote as $A$ the system matrix, as $x^*$ the true solution, as $x_j$ and $r_j$ the updated solution and the residual before the $j^{th}$ iteration, respectively, and as $\alpha_j$ and $p_j$ the step size and search direction in the $j^{th}$ iteration, respectively. These methods return the true solution $x^{*}$ in at most $N$ iterations, where $N$ is the dimension of the problem.

\subsubsection{Conjugate Gradient}
The conjugate gradient algorithm (CG) is an iterative method to solve symmetric and positive definite linear systems $A^{-1}b$. 
\begin{algorithm}
\caption{Conjugate Gradient}
\label{CG}
\begin{algorithmic}[1]
\STATE{Compute $r_0 := b - Ax_0$, $p_0 := r_0$}
\FOR{$j = 0, 1,\ldots,$ until convergence}
\STATE{$\alpha_j := (r_j, r_j) / (Ap_j, p_j) $}
\STATE{$x_{j+1} := x_j + \alpha_j p_j $}
\STATE{$r_{j+1} := r_j - \alpha_j A p_j $}
\STATE{$\beta_j := (r_{j+1}, r_{j+1}) / (r_j, r_j) $}
\STATE{$p_{j+1} := r_{j+1} + \beta_j p_j $}
\ENDFOR
\end{algorithmic}
\end{algorithm}
\newline
Each iteration minimizes the squared energy norm $\|x-x^*\|_A^2$ of the error. Since the vectors $\{p_j\}_{j = 0, 1, 2, \ldots, N-1}$ are $A$-orthogonal, the squared energy norm of the error before the $j^{th}$ iteration is
\begin{equation}
    \|x_j-x^*\|_A^2 = \left(\sum_{k=j}^{N-1} \alpha_k p_k\right)^{T} A \left(\sum_{k=j}^{N-1} \alpha_k p_k\right) = \sum_{k=j}^{N-1} \alpha_k^2 \|p_k\|_{A}^2.
\end{equation}
An optimal truncation estimator minimizes the variance of the solution for a given average number of iterations. In CG, we can measure the variance as the expected squared energy norm of the error. 
In the stochastic setting, let $\delta_j$ be the random variable that represents the step size in the $j^{th}$ iteration. 
It takes the value of $w_j \alpha_j$ or 0, where $w_j$ is the weight. We require $\mathbb{E}(\delta_j) = \alpha_j$ for an unbiased solution. 
Denoting the probability of running the $j^{th}$ iteration as $\mathbb{Q}(j)$, we have $w_j = \frac{1}{\mathbb{Q}(j)}$ and 
\begin{equation}
\begin{aligned}
\min_{\mathbb{Q}} \quad & \sum_{j = 0}^{N-1}  \mathbb{Q}(j) \left(\alpha_j - \frac{\alpha_j}{\mathbb{Q}(j)}\right)^2 \|p_j\|_{A}^2 + \left(1 - \mathbb{Q}(j)\right) \alpha_j^2 \|p_j\|_A^2  \\
=& \sum_{j = 0}^{N-1} \frac{1-\mathbb{Q}(j)}{\mathbb{Q}(j)}\alpha_j^2 \|p_j\|_A^2.
\end{aligned}
\end{equation}
The optimal solution to this problem is the Single Sample estimator \cite{10.1214/15-STS523}. However, performing a CG iteration inherently depends on the results of all preceding iterations, making it essential to compute them sequentially. Thus, we must ensure $\mathbb{Q}(k) \leq \mathbb{Q}(j)$ for $k>j$. Therefore, we optimize the probabilities $\mathbb{P}(j)$ of truncating before the $j^{th}$ iteration. 
This produces an adaptive Russian Roulette estimator \cite{P-766} that requires only one future CG iteration, resulting in minimal computational overhead, as detailed in \cref{sec: problemform,sec: solvingpro}.

\subsubsection{Generalized Minimal Residual}
The generalized minimal residual algorithm (GMRES) minimizes the residual norm over all vectors in $x_0 + \mathcal{K}_{m}$, where $\mathcal{K}_{m}$ is the $m$-th Krylov subpace. 
We present the most popular of its many variants.
\begin{algorithm}
\caption{GMRES}
\label{GMRES}
\begin{algorithmic}[1]
\STATE{Compute $r_0 := b - Ax_0$, $\beta := \|r_0\|_2$, and $v_1:= r_0/ \beta$}
\STATE{Initialize $V = [v_1]$}
\FOR{$j = 1, 2, \ldots, m$}
    \STATE{Compute $w_j := Av_j$}
    \FOR{$i = 1, 2, \ldots, j$}
        \STATE{$h_{i,j} := (w_j, v_i)$}
        \STATE{$w_j := w_j - h_{i,j} v_i$}
    \ENDFOR
    \STATE{$h_{j+1,j} := \|w_j\|_2$}
    \IF{$h_{j+1,j} = 0$}
        \STATE{set $m:=j$ and go to Line $16$}
    \ENDIF
    \STATE{$v_{j+1} := w_j / h_{j+1,j}$}
    \STATE{Append $v_{j+1}$ to $V$}
\ENDFOR
\STATE{Define the $(m+1) \times m$ Hessenberg matrix $H_m = \{h_{i,j}\}_{1\leq i \leq m+1, 1 \leq j \leq m} $}
\STATE{Find $y_m$ that minimizes $\| \beta e_1 - H_m y_m \|_2$ \COMMENT{Solve least squares problem}}\label{Line:iterate}
\STATE{$x_m := x_0 + V y_m$}
\end{algorithmic}
\end{algorithm}
This algorithm does not return intermediate solutions $x_j$, which are required by our randomized estimator. \cite{saad2003iterative} describes a modification of line $17$ that efficiently computes intermediate solutions on the fly. 
To simplify the exposition in this work, we will use the conjugate residual algorithm, a special case of GMRES.

\subsubsection{Conjugate Residual}
The conjugate residual algorithm (CR) is derived from GMRES applied to Hermitian matrices. 
Its structure resembles that of CG. 
\begin{algorithm}
\caption{Conjugate Residual}
\label{CR}
\begin{algorithmic}[1]
\STATE{Compute $r_0 := b - Ax_0$, $p_0 := r_0$}
\FOR{$j = 0, 1, \ldots,$ until convergence}
\STATE{$\alpha_j := (r_j, Ar_j) / (Ap_j, Ap_j) $}
\STATE{$x_{j+1} := x_j + \alpha_j p_j $}
\STATE{$r_{j+1} := r_j - \alpha_j A p_j $}
\STATE{$\beta_j := (r_{j+1}, Ar_{j+1}) / (r_j, Ar_j) $}
\STATE{$p_{j+1} := r_{j+1} + \beta_j p_j $}
\STATE{Compute $Ap_{j+1} = Ar_{j+1} + \beta_j A p_j$}
\ENDFOR
\end{algorithmic}
\end{algorithm}
\newline
Each iteration minimizes the squared $2$-norm $\|r\|^2_{2}$ of the residual. 
Since the vectors $\{Ap_j\}_{j = 0, 1, 2, \ldots, N-1}$ are orthogonal, the squared residual $2$-norm before the $j^{th}$ iteration is
\begin{equation}
    \|r_{j}\|_{2}^2 = \sum_{k=j}^{N-1} \alpha_k^2 \|A p_k\|_{2}^2.
\end{equation}
The derivation of our truncation estimator follows that for CG, with the squared energy norm of the error replaced by the squared $2$-norm of the residual.

\subsection{The optimization problem} \label{sec: problemform}
%
We denote as $\bar{x}_j$ the adjusted solution before the $j^{th}$ iteration in the stochastic setting. In CG, the adjusted squared energy norm of the error before the $j^{th}$ iteration by Russian Roulette estimator is
\begin{equation}    
    \begin{aligned} 
    \|\bar{x}_j - x^*\|_{A}^2 =  \sum_{k = 0}^{N-1} \alpha_k^2 \|p_k\|_{A}^2, \quad &j = 0 \\
    \|\bar{x}_j - x^*\|_{A}^2 =  \sum_{k = 0}^{j-1} \left(\alpha_{k} -  \frac{\alpha_{k}}{1- \sum_{i = 0}^{k}\mathbb{P}(i)}\right)^2 \|p_k\|_{A}^2 + \sum_{k = j}^{N-1} \alpha_k^2 \| p_k\|_{A}^2, \quad &1 \leq j \leq N-1
    \end{aligned} 
\end{equation}
The adjusted squared energy norm of the error after all $N$ iterations is
\begin{equation}    
    \|\bar{x}_N - x^*\|_{A}^2 = \sum_{k = 0}^{N-1} \left(\alpha_{k} -  \frac{\alpha_{k}}{1- \sum_{i = 0}^{k}\mathbb{P}(i)}\right)^2 \|p_k\|_{A}^2 
\end{equation}
where $\mathbb{P}(j)$ is the probability of truncating before the $j^{th}$ iteration. We define $\mathbb{P}(N)$ as the probability of implementing all $N$ iterations.

We denote as $\bar{r}_j$ the adjusted residual before the $j^{th}$ iteration in the stochastic setting. In CR, the adjusted residual $2$-norm squared before the $j^{th}$ iteration by Russian Roulette estimator is
\begin{equation}    
\begin{aligned}
    \|\bar{r}_j\|_{2}^2 = \sum_{k = 0}^{N-1} \alpha_k^2 \|A p_k\|_{2}^2, \quad &j = 0 \\
    \|\bar{r}_j\|_{2}^2 =  \sum_{k = 0}^{j-1} \left(\alpha_{k} - \frac{\alpha_{k}}{1- \sum_{i = 0}^{k}\mathbb{P}(i)}\right)^2 \|A p_k\|_{2}^2 + \sum_{k = j}^{N-1} \alpha_k^2 \|A p_k\|_{2}^2, \quad &1 \leq j \leq N-1.
\end{aligned} 
\end{equation}
The adjusted residual $2$-norm squared after all $N$ iterations is
\begin{equation}    
    \|\bar{r}_N\|_{2}^2 = \sum_{k = 0}^{N-1} \left(\alpha_{k} - \frac{\alpha_{k}}{1- \sum_{i = 0}^{k}\mathbb{P}(i)}\right)^2 \|A p_k\|_{2}^2. 
\end{equation}
\Cref{tab:notation} introduces notation for treating the cases of CG and CR simultaneously.
\begin{table}[htbp]
{\footnotesize
  \caption{Shared notation for CG and CR}  \label{tab:notation}
\begin{center}
  \begin{tabular}{|c|c|c|} \hline
   \bf Notation & \bf CG & \bf CR \\ \hline
    $\|v_j\|$ & $\|\bar{x}_j - x^*\|_{A}$ & $\|\bar{r}_j\|_{2}$ \\
    $\|q_j\|$ & $\|p_j\|_{A}$ & $\|Ap_j\|_{2}$ \\ \hline
  \end{tabular}
\end{center}
}
\end{table}
For a fixed average number of iterations $C$ such that $0 < C \leq N$, variance minimization amounts to

\begin{equation} \label{eq: mini-prob}
\begin{aligned}
\min_{\mathbb{P}} \quad & \sum_{j = 0}^{N}  \mathbb{P}(j) \|v_{j}\|^2 \\
\textrm{s.t.} \quad & \sum_{j = 0}^{N} j \mathbb{P}(j) = C \\
  & \sum_{j = 0}^{N} \mathbb{P}(j) = 1\\
  & -\mathbb{P}(j) \leq 0, \quad 0 \leq j \leq N,
\end{aligned}
\end{equation}
where $\|v_j\|^2$ satisfies
\begin{equation} \label{eq: v condition}
    \|v_j\|^2 = \begin{dcases}
        \sum_{k = j}^{N-1} \alpha_k^2 \|q_k\|^2, \quad &j = 0 \\
        \sum_{k = 0}^{j-1} \left(\alpha_{k} - \frac{\alpha_{k}}{1- \sum_{i = 0}^{k}\mathbb{P}(i)}\right)^2 \|q_k\|^2 + \sum_{k = j}^{N-1} \alpha_k^2 \|q_k\|^2, \quad &1 \leq j \leq N-1 \\
        \sum_{k = 0}^{j-1} \left(\alpha_{k} - \frac{\alpha_{k}}{1- \sum_{i = 0}^{k}\mathbb{P}(i)}\right)^2 \|q_k\|^2, \quad &j = N.
    \end{dcases}
\end{equation}
The Lagrangian of the problem is
\begin{equation} 
\mathcal{L}= \sum_{j = 0}^{N}  \mathbb{P}(j) \|v_{j}\|^2 + \lambda\left(C- \sum_{j = 0}^{N} j \mathbb{P}(j)\right) + \mu \left(1- \sum_{j = 0}^{N} \mathbb{P}(j)\right) + \sum_{j=0}^{N}u_j\left(-\mathbb{P}(j)\right),
\end{equation}
and the KKT conditions are: 
\begin{equation}\label{eq: KKT cond}
\begin{aligned}
u_j \left(-\mathbb{P}(j)\right) &= 0, \\
u_j &\geq 0,
\end{aligned}
\end{equation}
where $0 \leq j \leq N$.

\subsection{Solving the optimization problem} \label{sec: solvingpro} 
To simplify the first-order optimality condition, we define 
\begin{equation}
    m_j := \left(\alpha_{j} - \frac{\alpha_{j}}{1- \sum_{i = 0}^{j}\mathbb{P}(i)}\right)\left(-\frac{\alpha_{j}}{(1- \sum_{i = 0}^{j}\mathbb{P}(i))^2}\right) \|q_{j}\|^2.
\end{equation}
The first-order conditions imply that 
\begin{equation}
    \frac{\partial\mathcal{L}}{\partial \mathbb{P}(j)} = \begin{dcases}
        \|v_j\|^2 + 2 \sum_{k=(j+1)}^{N}\mathbb{P}(k)\left(\sum_{i = j}^{k-1} m_i \right) - j \lambda  - \mu -u_j = 0, \quad &0 \leq j \leq N-1 \\
        \|v_j\|^2 - j \lambda  - \mu -u_j = 0, \quad &j = N.
    \end{dcases}
\end{equation}
Since $\mu$ is freely adjustable, we assume that $\frac{\partial\mathcal{L}}{\partial \mathbb{P}(0)} = 0$.
Satisfying the other first-order conditions requires that
\begin{equation}\label{eq: foc difference}
    \frac{\partial\mathcal{L}}{\partial \mathbb{P}(j)} - \frac{\partial\mathcal{L}}{\partial \mathbb{P}(j+1)} = 0, \quad 0 \leq j \leq N-1,
\end{equation}
which is equivalent to
\begin{equation}\label{eq: Diff}
    \left(\|v_j\|^2 - \|v_{j+1}\|^2\right) + 2\sum_{k=j+1}^{N}\mathbb{P}(k) m_j + \lambda - u_j + u_{j+1} = 0, \quad 0 \leq j \leq N-1.
\end{equation}
We can simplify \eqref{eq: Diff} to
\begin{equation}\label{eq: Diff2}
    \frac{\alpha_{j} \|q_{j}\|}{1-\sum_{k=0}^{j}\mathbb{P}(k)} = \sqrt{-\lambda + u_j - u_{j+1}}, \quad 0 \leq j \leq N-1.
\end{equation}
For a given average number of iterations $C$, there must be an integer $-1 \leq n \leq N-1$ such that
\begin{equation}
        \mathbb{P}(j) = 0, \quad 0 \leq j \leq n,
\end{equation}
and $\mathbb{P}(n+1)$ is the first non-zero probability. If $n = N-1$, $\mathbb{P}(N) = 1$ and the estimator results in the deterministic iterative method. Therefore, we only consider $-1 \leq n \leq N-2$ from then on.

\begin{lemma}
Assume that the Krylov subspace iterative method has a diminishing returns property, i.e., $\alpha_{j-1}^2 \|q_{j-1}\|^2 > \alpha_{j}^2 \|q_{j}\|^2$ for every j. Then, 
\begin{equation}
        \mathbb{P}(j) > 0, \quad n+1 \leq j \leq N.
\end{equation}
\end{lemma}

\begin{proof}
First, we consider the special case of $n = N-2$. Since the estimator is unbiased, $0<\mathbb{P}(N-1)<1$. Thus, $\mathbb{P}(N) = 1 - \mathbb{P}(N-1) > 0$. 
Given $-1 \leq n < N-2$, we assume that there exist $m$ and $s$ such that $1 \leq m < s < N-n$ and
    \begin{equation}
        \begin{aligned}
            \mathbb{P}(j) > 0, \quad &n+1 \leq j \leq n+m \\
            \mathbb{P}(j) = 0, \quad &n+m+1 \leq j \leq n+s \\
            \mathbb{P}(j) > 0, \quad &j = n+s+1.
        \end{aligned}
    \end{equation}
By the KKT conditions,
    \begin{equation}
        \begin{aligned}
            u_{n+m} &= 0 \\
            u_{n+s+1} &= 0.
        \end{aligned}
    \end{equation}
Therefore,
    \begin{equation}
       \begin{aligned}
            \frac{\alpha_{n+m} \|q_{n+m}\|}{1-\sum_{k=0}^{n+m}\mathbb{P}(k)}
            &= \sqrt{-\lambda + u_{n+m} - u_{n+m+1}} \\
            &\leq \sqrt{-\lambda + u_{n+s} - u_{n+s+1}}\\
            &= \frac{\alpha_{n+s} \|q_{n+s}\|}{1-\sum_{k=0}^{n+s}\mathbb{P}(k)}, 
        \end{aligned} 
    \end{equation}
which violates the diminishing returns property of $\alpha_{n+m}^2 \|q_{n+m}\|^2 > \alpha_{n+s}^2 \|q_{n+s}\|^2$.
Thus, the proof is complete.
\end{proof}
Since $\mathbb{P}(j) >0$ for $n+1 \leq j \leq N$,
\begin{equation}
    \begin{aligned}
        u_j &= 0, \quad n+1 \leq j \leq N \\
        u_j &\geq 0, \quad \text{otherwise}.
    \end{aligned}
\end{equation}
\eqref{eq: Diff2} becomes
\begin{equation} \label{eq: KKT eq}
    \begin{aligned}
    \frac{\alpha_j \|q_j\|}{1-\sum_{k=n+1}^{j}\mathbb{P}(k)} = \sqrt{-\lambda}, \quad &n+1 \leq j \leq N-1 \\
    \alpha_{j} \|q_{j}\| = \sqrt{-\lambda + u_j}, \quad &n \geq 0 \; \text{and} \; j = n\\
    \alpha_j \|q_j\| = \sqrt{-\lambda + u_j - u_{j+1}}, \quad &\text{otherwise}
    \end{aligned}
\end{equation}
For the first positive truncation probability, we have
\begin{equation}
    \mathbb{P}(n+1) = 1 - \frac{\alpha_{n+1}\|q_{n+1}\|}{\sqrt{-\lambda}}.
\end{equation}
By assigning the initial truncation probability $\mathbb{P}(n+1)$, we adjust $\lambda$ implicitly. If $n = -1$, there is no upper bound for the choice of $\mathbb{P}(n+1)$, i.e., $\mathbb{P}(n+1)\in (0,1)$. If $n \geq 0$, to satisfy the KKT conditions that $u_{n} \geq 0$ the maximum value we can choose for $\mathbb{P}(n+1)$ is
\begin{equation}
    \mathbb{P}(n+1) = \frac{\alpha_{n} \|q_{n}\| - \alpha_{n+1} \|q_{n+1}\|}{\alpha_{n} \|q_{n}\|}.
\end{equation}
All consequent truncation probabilities are

\begin{equation}
        \mathbb{P}(j) = \begin{dcases}
            \frac{\left(\alpha_{j-1} \|q_{j-1}\| - \alpha_{j} \|q_{j}\| \right) \left(1-\mathbb{P}(n+1) \right)}{\alpha_{n+1} \|q_{n+1}\|} , \quad &n+1 < j \leq N-1 \\
            \frac{\left(\alpha_{j-1} \|q_{j-1}\| \right) \left(1-\mathbb{P}(n+1) \right)}{\alpha_{n+1} \|q_{n+1}\|} , \quad &j = N.
        \end{dcases}
\end{equation}
Computing the truncation probabilities requires only one future iteration. Without information about all iterations, it is impossible to specify the variance or average number of iterations explicitly. However, we can control the trade-off between output variance and computational cost by enforcing a certain number of deterministic iterations and by choosing the initial truncation probability. 
\begin{definition}\label{def: original AS}
The AS estimator requires a continuous parameter $\eta \in (-1, N-1)$, with $n = \lfloor \eta \rfloor$, and $\sigma = \eta - n$. All truncation probabilities are:

\begin{equation}
        \mathbb{P}(j) = \begin{dcases}
        0, \quad &0 \leq j \leq n \\
            1-\sigma, \quad &n = -1 \; \mathrm{and} \; j = n+1 \\
            \left(1-\sigma \right) \left(\frac{\alpha_{n} \|q_{n}\| - \alpha_{n+1} \|q_{n+1}\|}{\alpha_{n} \|q_{n}\|} \right), \quad &n \geq 0 \; \mathrm{and} \; j = n+1 \\
         \frac{\left(\alpha_{j-1} \|q_{j-1}\| - \alpha_{j} \|q_{j}\| \right) \left(1-\mathbb{P}(n+1) \right)}{\alpha_{n+1} \|q_{n+1}\|}, \quad &n+1 < j \leq N-1 \\
        \frac{\left(\alpha_{j-1} \|q_{j-1}\| \right) \left(1-\mathbb{P}(n+1) \right)}{\alpha_{n+1} \|q_{n+1}\|}, \quad &j = N.
        \end{dcases}
\end{equation}
The average number of iterations is
\begin{equation}
    C = \begin{dcases}
        \sigma+\sigma \frac{\sum_{j = n+2}^{N-1}\alpha_{j} \|q_{j}\|}{\alpha_{n+1} \|q_{n+1}\|}, \quad &n = -1 \\
        n + 1 + \left(\sigma + \left(1-\sigma \right) \frac{\alpha_{n+1} \|q_{n+1}\|}{\alpha_{n} \|q_{n}\|}\right)\left(1+\frac{\sum_{j = n+2}^{N-1}\alpha_{j} \|q_{j}\|}{\alpha_{n+1} \|q_{n+1}\|}\right), \quad &\text{otherwise}.
    \end{dcases}
\end{equation}
\end{definition}

\begin{theorem}
    Assume that the Krylov subspace iteration method has a diminishing returns property. For a given average number of iterations $C$, the AS estimator returns the minimum solution to \eqref{eq: mini-prob}.
\end{theorem}

\begin{proof}
    First, consider the special case of $C = N$, which leads to
    \begin{equation} \label{eq: special point}
        \mathbb{P}(j) = \begin{dcases}
            0, \quad &0 \leq j \leq N-1 \\
            1, \quad &j = N.
        \end{dcases}
    \end{equation}
    The AS estimator returns the same truncation probabilities as in \eqref{eq: special point}, and the objective of the problem becomes $0$ which is the minimum.
    
    For $0 < C < N$, if $\mathbb{P}(N) \rightarrow 0$, $\sum_{j=0}^{N-1}\mathbb{P}(j) \rightarrow 1$ and the objective function $\sum_{j = 0}^{N}  \mathbb{P}(j) \|v_{j}\|^2 \rightarrow \infty$ as indicated in \eqref{eq: v condition}. Thus, we can assume a lower bound on $\mathbb{P}(N)$. For every set of the form $\{\mathbb{P}(j) : \sum_{j = 0}^{N} \mathbb{P}(j) = 1 \; \text{and} \; \mathbb{P}(N) \geq \epsilon\}$ for some $\epsilon > 0$, the set is compact since it is a closed and bounded subset of $\mathbb{R}^N$. As the objective function is continuous on this compact set, it has a minima. For $\epsilon$ small enough, the AS estimator covers all possible values of $0 < C < N$ by adjusting the value of $\eta$ and returns the only point that satisfies the first-order condition. Therefore, the solution returns the minimum given the domain $\{\mathbb{P}(j) : \sum_{j = 0}^{N} \mathbb{P}(j) = 1 \; \text{and} \; \mathbb{P}(N) \geq \epsilon\}$. This minimum is also smaller than any solution generated by truncation probabilities outside of the set because $\mathbb{P}(N) \rightarrow 0$ (and thus, $\epsilon \rightarrow 0$) leads to the objective blowing up. Thus, it is the optimal solution to the problem \eqref{eq: mini-prob}.
\end{proof}

\subsection{Without diminishing returns}
In \cref{sec: solvingpro}, we require that the iterative method satisfies a diminishing returns property, i.e., $\alpha_{j-1}^2 \|q_{j-1}\|^2 > \alpha_{j}^2 \|q_{j}\|^2$ for every j. However, this property does not generally hold for CG or CR. We show that when the diminishing returns property fails, it is impossible to obtain the optimal on-the-fly truncation distribution. To avoid the bias introduced by negative truncation probabilities, we propose a generalized AS estimator. 
\begin{theorem}\label{thm: unpracticality}
Given $-1 \leq n < N-3$, $\mathbb{P}(j) = 0$ for $0 \leq j \leq n$, and $\mathbb{P}(n+1) = c > 0$ for 
\begin{equation}
    c \in \begin{dcases}
        (0, 1), \quad &n = -1 \\
        (0, \frac{\alpha_{n} \|q_{n}\| - \alpha_{n+1} \|q_{n+1}\|}{\alpha_{n} \|q_{n}\|}), \quad &\text{otherwise}
    \end{dcases}
\end{equation}
Suppose that after $m$ additional iterations we have computed $\mathbb{P}(n+2)$ as the components for the optimal point, where $n+m \leq N-2$. There are certain iteration components that cause the computed probabilities to lose their optimality.
\end{theorem}

\begin{proof}
    Fixed $\epsilon > 0$, construct an iteration update
    \begin{equation} \label{eq: severalcon}
    \begin{aligned}
            \alpha_{n}^2 \|q_{n}\|^2 &> \alpha_{n+1}^2 \|q_{n+1}\|^2 \\
            \alpha_{n+2}^2 \|q_{n+2}\|^2 + \epsilon &> \alpha_{n+1}^2 \|q_{n+1}\|^2 > \alpha_{n+2}^2 \|q_{n+2}\|^2 \\
        \alpha_{j+1}^2 \|q_{j+1}\|^2 &= \alpha_{j}^2 \|q_{j}\|^2, \quad m \geq 2 \; \text{and} \; n+2 \leq j < n + m +1 \\
        \sum_{j=n+m+2}^{N-1}\alpha_j^2 \|q_j\|^2 &> \alpha_{n+1}^2 \|q_{n+1}\|^2 + m \epsilon.
    \end{aligned}
    \end{equation}
    
    \begin{itemize}
        \item{Case 1: $\mathbb{P}(n+2) = 0$} \\
        The diminishing returns property holds from the $(n+1)^{th}$ iteration to the $(n+m + 1)^{th}$ iteration. We can generate future iteration updates by ensuring that this property holds throughout the iteration, i.e., $\alpha_{j-1}^2 \|q_{j-1}\|^2 > \alpha_{j}^2 \|q_{j}\|^2$ for every $j$. The result returned by the AS estimator is the optimal point as established in \cref{def: original AS}. Thus, the value of $\mathbb{P}(n+2)$ must be
        \begin{equation}
            \mathbb{P}(n+2) = \frac{\left(\alpha_{n+1} \|q_{n+1}\| - \alpha_{n+2} \|q_{n+2}\|\right) \left(1-\mathbb{P}(n+1)\right)}{\alpha_{n+1} \|q_{n+1}\|} > 0.
        \end{equation}
        $\mathbb{P}(n+2) = 0$ is not optimal in this case.\\

        \item{Case 2: $\mathbb{P}(n+2) > 0$} \\     
         We can create the iteration update in the $(n+m+2)^{th}$ iteration as
         \begin{equation}\label{eq: lastcon}
        \begin{aligned}
        \alpha_{n+m+2}^2 \|q_{n+m+2}\|^2 = \alpha_{n+1}^2 \|q_{n+1}\|^2 + m\epsilon.
        \end{aligned}
        \end{equation}
        By \eqref{eq: severalcon} and \eqref{eq: lastcon}, 
        \begin{equation}
            \left(m+1\right)\alpha_{n+1}^2 \|q_{n+1}\|^2 < \sum_{j=n+2}^{n+m+2}\alpha_j^2 \|q_j\|^2.
        \end{equation}
        Since $\mathbb{P}(n+1) > 0$ and $\mathbb{P}(n+2) > 0$,  $u_{n+1} = 0$ and $u_{n+2} = 0$ based on KKT conditions.
        By \eqref{eq: Diff2},
        \begin{equation}
        \begin{aligned}
            \left(m+1\right)(-\lambda)
            &= \left(m+1\right)\left(-\lambda + u_{n+1} -u_{n+2}\right)\\
            &=\left(m+1\right) \frac{\alpha_{n+1}^2 \|q_{n+1}\|^2}{\left(1-\sum_{k=0}^{n+1}\mathbb{P}(k)\right)^2}\\ 
            &< \sum_{j=n+2}^{n+m+2} \frac{\alpha_j^2 \|q_j\|^2}{\left(1-\sum_{k=0}^{n+m+2}\mathbb{P}(k)\right)^2}\\
            &= \sum_{j=n+2}^{n+m+2} \left(-\lambda+u_j-u_{j+1}\right)\\
            &= \left(m+1\right)\left(-\lambda \right) - u_{n+m+3}.
        \end{aligned} 
        \end{equation} 
        The result contradicts the condition in \eqref{eq: KKT cond}, which requires all multipliers to be non-negative. Thus, $\mathbb{P}(n+2) > 0$ is not optimal in this case.
    \end{itemize} 
    
    Thus, the proof is complete.   
\end{proof}
    As demonstrated in \cref{thm: unpracticality}, the AS estimator may be biased absent the diminishing returns property. To address this issue, we propose the generalized AS estimator, introduced in \cref{def: AS estimator}, which returns the optimal truncation estimator under the diminishing returns property. When the condition fails, it offers a suboptimal yet unbiased solution, ensuring robustness across different cases.


\subsection{Generalized AS estimator}

\begin{definition} \label{def: AS estimator}
    The AS estimator requires a parameter $\eta \in (-1, N-1)$, with $n = \lfloor \eta \rfloor$, and $\sigma = \eta - n$. If $n \geq 0$, we assume that $\alpha_{n}^2 \|q_{n}\|^2 > \alpha_{n+1}^2 \|q_{n+1}\|^2$. If the assumption fails, we set $\mathbb{P}(n+1) = 0$. All truncation probabilities before the $(n+1)^{th}$ iteration are
    \begin{equation}
        \mathbb{P}(j) = \begin{dcases}
            0, \quad  &0 \leq j \leq n \\
            1-\sigma, \quad &n = -1 \; \mathrm{and} \; j = n+1 \\
            \max\{0, \left(1-\sigma \right)\left(\frac{\alpha_{n} \|q_{n}\| - \alpha_{n+1} \|q_{n+1}\|}{\alpha_{n} \|q_{n}\|}\right)\}, \quad &n \geq 0 \; \mathrm{and} \; j = n+1.
            \end{dcases}
    \end{equation}
    The AS estimator groups the iteration terms into different sets: 
    \begin{equation}
        \begin{aligned}
            S_0 &= \{n+1\} \\
            S_1 &= \{n+2, \ldots, n+2+n_1 \},  \\
            S_2 &= \{n+3+n_1, \ldots, n+3+n_1+n_2 \} \\
            &\vdots \\
            S_m &= \{n+m+1+n_1+n_2+, \ldots, + n_{m-1}, \ldots, n+m+1+n_1+n_2+, \ldots, + n_m\}
        \end{aligned} 
    \end{equation} 
    where each $n_k \geq 0$ and each $n_k$ is the smallest value that makes
    \begin{equation}
        \frac{\sum_{i\in S_{k-1}} \alpha_{i}^2 \|q_{i}\|^2}{|S_{k-1}|} \geq \frac{\sum_{i \in S_{k}} \alpha_{i}^2 \|q_{i}\|^2}{|S_{k}|}, \quad k \geq 1
    \end{equation} 
    valid. The last set $S_m$ also satisfies that
    \begin{equation}
    \begin{aligned}
        n+m+1+n_1+n_2+, \ldots, + n_m &\leq N-1 \\
        n+m+1+n_1+n_2+, \ldots, + n_{m-1} &\geq N-1
    \end{aligned}
    \end{equation}
    Note that $(n+m+1+n_1+n_2+, \ldots, + n_m)$ may not always be equal to $(N-1)$. In case $(n+m+1+n_1+n_2+, \ldots, + n_m) > N-1$, we assume that there are extra dummy iterations which do not improve the solution, i.e., 
    \begin{equation}
        \alpha_j^2\|q_j\|^2 = 0, \quad N \leq j \leq n+m+1+n_1+n_2+, \ldots, + n_m
    \end{equation} 
    For a fixed $S_k$, define $g_k := \frac{\sum_{i \in S_{k}} \alpha_{i}^2 \|q_{i}\|^2}{|S_{k}|}$. The AS estimator returns the truncation probabilities after the $(n+1)^{th}$ iteration
    \begin{equation} \label{eq: P(j) AS-CG}
            \mathbb{P}(j) = \begin{dcases}
                     \left(1-\mathbb{P}(n+1)\right) \frac{\sqrt{g_{k-1}} - \sqrt{g_k}}{\alpha_{n+1} \|q_{n+1}\|}, \quad &j = \min_{i \in S_k} i \\
                     \left(1-\mathbb{P}(n+1)\right) \frac{\sqrt{g_m}}{\alpha_{n+1} \|q_{n+1}\|}, \quad & j = N \\
                    0, \quad & \mathrm{otherwise}.
                 \end{dcases}
    \end{equation} 
\end{definition}
%
We group iteration terms into different sets, ensuring diminishing returns at the set level. We assign the same survival probability for all iteration terms in each set. \cref{algo: Truncation Probabilities Solver} outlines the formal solving procedure.
\begin{algorithm}
\caption{AS Estimator($\eta$)}
\label{algo: Truncation Probabilities Solver}
\begin{algorithmic}[1]
\STATE{Define $\mathbb{P}:= \{(0), (1), (2), \ldots, (N) $\}}
\STATE{Define $n:= \lfloor \eta \rfloor$, $\sigma := \eta - n$}
\IF{$n = -1$}
    \STATE{$\mathbb{P}(n+1) = 1 - \sigma$}
\ELSE
    \FOR{$j = 0 : n$}
        \STATE{$\mathbb{P}(j) := 0$}
    \ENDFOR
    \STATE{$\mathbb{P}(n+1) = \max\{0, \left(1-\sigma \right)\left(\frac{\alpha_{n} \|q_{n}\| - \alpha_{n+1} \|q_{n+1}\|}{\alpha_{n} \|q_{n}\|}\right)\}$}
    \IF{$n = N-2$}
        \STATE{$\mathbb{P}(N) = 1- \mathbb{P}(n+1)$}
    \ENDIF
\ENDIF
\STATE{Define $g_{\mathrm{prev}}:= \alpha_{n+1}^2 \|q_{n+1}\|^2$; $g_{\mathrm{curr}}:= \alpha_{n+2}^2 \|q_{n+2}\|^2$; $\text{count} := 1$ $i := n+3$}
\WHILE{$i - \text{count} < N$}
\IF{$g_{\mathrm{prev}} < g_{\mathrm{curr}}$}
    \IF{$i \geq N$}
        \STATE{Update $g_{\mathrm{curr}}:= \left(g_{\mathrm{curr}} * \text{count}\right)/ \left(\text{count} + 1\right)$}
    \ELSE
        \STATE{Update $g_{\mathrm{curr}}:= \left(g_{\mathrm{curr}} * \text{count} + \alpha_{i}^2 \|q_{i}\|^2\right)/ \left(\text{count} + 1\right)$}
    \ENDIF
    \STATE{\text{count} = \text{count} + 1}
    \STATE{$i = i + 1$}
\ELSE
    \STATE{$\mathbb{P}(i - \text{count}) := \frac{\sqrt{g_{\mathrm{prev}}} - \sqrt{g_{\mathrm{curr}}}}{\alpha_{n+1} \|q_{n+1}\|} * \left(1 - \mathbb{P}(n+1)\right)$}
    \IF{$\text{count} > 1$}
        \FOR{$k = i-\text{count} + 1 : \max\{i - 1, N-1\}$}
            \STATE{$\mathbb{P}(k) := 0$}
        \ENDFOR
    \ENDIF
    \STATE{$g_{\mathrm{prev}} := g_{\mathrm{curr}}$;  
    $i = i + 1$;
    $\text{count} = 1$}
    \IF{$i < N$}
        \STATE{$g_{\mathrm{curr}}:= \alpha_{i}^2 \|q_{i}\|^2$}
    \ENDIF
\ENDIF
\ENDWHILE
\STATE{$\mathbb{P}(N) = \frac{\sqrt{g_{\mathrm{curr}}}}{\alpha_{n+1} \|q_{n+1}\|} * \left(1 - \mathbb{P}(n+1)\right)$}
\RETURN {$\mathbb{P}$}
\end{algorithmic}
\end{algorithm}



To numerically solve complex systems of equations, the Krylov iterative method does not reach the exact solution after $N$ iterations due to the floating-point rounding error. We modify the AS estimator for implementation so as to continuously group iteration terms into different sets beyond the first $N$ iterations. In \cref{def: AS estimator} for any $S_k$ such that $k \geq 1$, the AS estimator requires all the iteration information in $S_k$ to compute $\mathbb{P}(j)$, where $j = \min_{i \in S_k} i$. In practice, we can reduce the solution variance while keeping the computational cost fixed by assigning the calculated probability to $\mathbb{P}(j)$, where $j = \max_{i \in S_k}i$. All the truncation probabilities after the $(n+1)^{th}$ iteration are
\begin{equation} \label{eq: modified P(j) AS-CG}
            \mathbb{P}(j) = \begin{dcases}
                     \left(1-\mathbb{P}(n+1)\right) \frac{\sqrt{g_{k-1}} - \sqrt{g_k}}{\alpha_{n+1} \|q_{n+1}\|}, \quad &j = \max_{i \in S_k}i \\
                     0, \quad &\text{otherwise}.
                 \end{dcases}
\end{equation}
\cref{algo: Practical implementation} shows the practical implementation. 
\begin{algorithm}
\caption{AS Estimator for Practical Implementation($\eta$)}
\label{algo: Practical implementation}
\begin{algorithmic}[1]
\STATE{Define $\mathbb{P}:= \{(0), (1), (2), \ldots, (\infty) $\}}
\STATE{Define $n:= \lfloor \eta \rfloor$, $\sigma := \eta - n$}
\IF{$n = -1$}
    \STATE{$\mathbb{P}(n+1) = 1 - \sigma$}
\ELSE
    \FOR{$j = 0 : n$}
        \STATE{$\mathbb{P}(j) := 0$}
    \ENDFOR
    \STATE{$\mathbb{P}(n+1) = \max\{0, \left(1-\sigma \right)\left(\frac{\alpha_{n} \|q_{n}\| - \alpha_{n+1} \|q_{n+1}\|}{\alpha_{n} \|q_{n}\|}\right)\}$}
\ENDIF
\STATE{Define $g_{\mathrm{prev}}:= \alpha_{n+1}^2 \|q_{n+1}\|^2$; $g_{\mathrm{curr}}:= \alpha_{n+2}^2 \|q_{n+2}\|^2$; $\text{count} = 1$}
\FOR{$i = n+3 : \infty$}

\IF{$g_{\mathrm{prev}} < g_{\mathrm{curr}}$}
    \STATE{Update $g_{\mathrm{curr}}:= \left(g_{\mathrm{curr}} * \text{count} + \alpha_{i}^2 \|q_{i}\|^2\right)/ \left(\text{count} + 1\right)$}
    \STATE{\text{count} = \text{count} + 1}
\ELSE
    \IF{$\text{count} > 1$}
        \FOR{$k = i-\text{count} : i - 2$}
            \STATE{$\mathbb{P}(k) := 0$}
        \ENDFOR
    \ENDIF
    \STATE{$\mathbb{P}(i - 1) := \frac{\sqrt{g_{\mathrm{prev}}} - \sqrt{g_{\mathrm{curr}}}}{\alpha_{n+1} \|q_{n+1}\|} * \left(1 - \mathbb{P}(n+1)\right)$}
    \STATE{$g_{\mathrm{prev}} := g_{\mathrm{curr}}$; $g_{\mathrm{curr}}:= \alpha_{i}^2 \|q_{i}\|^2$; $\text{count} = 1$}
\ENDIF
\ENDFOR
\RETURN {$\mathbb{P}$}
\end{algorithmic}
\end{algorithm}

\section{Experimental Results}
We show that our AS-CG estimator improves the trade-off between the solution variance and computational cost compared to the existing RR-CG method. 
We also compare models optimized with AS-CG against those with RR-CG and deterministic iterative methods across GP hyperparameter training and competitive physics-informed neural networks problem.

\subsection{Speed-Variance Trade-off Performance}
We conduct an experiment to compare the speed-variance trade-off between AS-CG and RR-CG in solving a single linear system. The matrix size is set to $500$, generated as a combination of a sparse random matrix with a density of $0.16$. Its non-zero entries follow a Gaussian distribution. To ensure the matrix is symmetric positive definite, we fill the diagonal with a constant value of $10.0$ and multiply the matrix by its transpose. The linear system converges after $284$ iterations with a tolerance of $1 \times 10^{-8}$ for the relative residual norm. For both AS-CG and RR-CG, we perform $300,000$ trials using a fixed set of parameters to solve the system. \cref{fig:same} shows the result of the experiment.


From \cref{fig:same}, we observe that the AS-CG estimator achieves a more favorable speed-variance trade-off than the RR-CG estimator. Although RR-CG shows improved performance as the temperature parameter $\lambda$ decreases from $0.10$ to $0.05$, it deteriorates when further reduced to $0.02$. The result suggests that, within a fixed computational budget, AS-CG provides a more reliable estimator to produce solutions with lower variance.

We extend the problem to multiple linear systems setting by introducing two distinct matrices. To create a more challenging linear system, we decrease the matrix diagonal to $8.0$, resulting in a system that requires $636$ iterations to converge. We construct an easier system by increasing the matrix diagonal to $13.0$, reducing the number of iterations to $91$ for convergence. We perform $300,000$ test for each of the three linear systems. \cref{fig:diff} shows the result of the experiment.

\cref{fig:diff} shows that AS-CG demonstrates a greater advantage in the multiple linear systems setting. In the harder linear system, where the CG error reduction deviates significantly from exponential decay, the RR-CG estimators with $\lambda = 0.10$ and $\lambda = 0.05$ exhibit oscillatory behavior. This suggests that increasing the computational cost may lead to higher variance. 

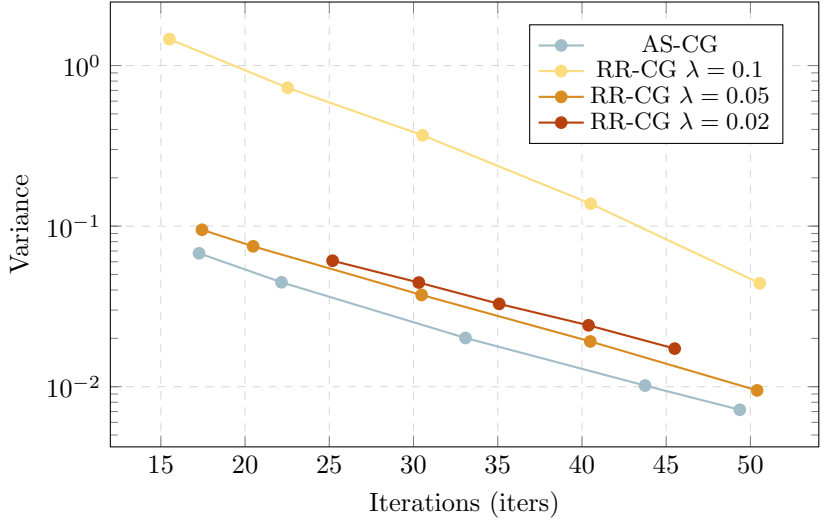
\begin{figure}
    \centering
    \begin{tikzpicture}
    \begin{axis}[
        width=11cm, height=7.5cm, 
        xlabel={Iterations (iters)},
        ylabel={Variance},
        ymode=log, 
        grid=major,
        grid style={dashed, gray!30},
        legend style={at={(0.95,0.95)},anchor=north east, , font=\small, row sep=-2pt},
    ]

    \addplot[
        mark=*,
        color=steelblue,
        thick, 
    ] table[x=iters, y=variance, col sep=comma] {csv_for_plot/test_acg_same_latest.csv};
    \addlegendentry{AS-CG}

    \addplot[
        mark=*,
        color=joshua,
        thick, 
    ] table[x=iters, y=variance, col sep=comma] {csv_for_plot/test_rrcg_same_latest_01.csv};
    \addlegendentry{RR-CG $\lambda = 0.1$}

    \addplot[
        mark=*,
        color=orange,
        thick, 
    ] table[x=iters, y=variance, col sep=comma] {csv_for_plot/test_rrcg_same_latest_005.csv};
    \addlegendentry{RR-CG $\lambda = 0.05$}

    \addplot[
        mark=*,
        color=rust,
        thick, 
    ] table[x=iters, y=variance, col sep=comma] {csv_for_plot/test_rrcg_same_latest_002.csv};
    \addlegendentry{RR-CG $\lambda = 0.02$}

    \end{axis}
\end{tikzpicture}
    \caption[Single Linear System]{Single Linear System: the variance with respect to the average number of CG iterations. The temperature parameter $\lambda$ is fixed for RR-CG in each line. For both methods, We vary the minimal truncation number and the initial truncation probability to change the variance and the average number of CG iterations.}
    \label{fig:same}
\end{figure}

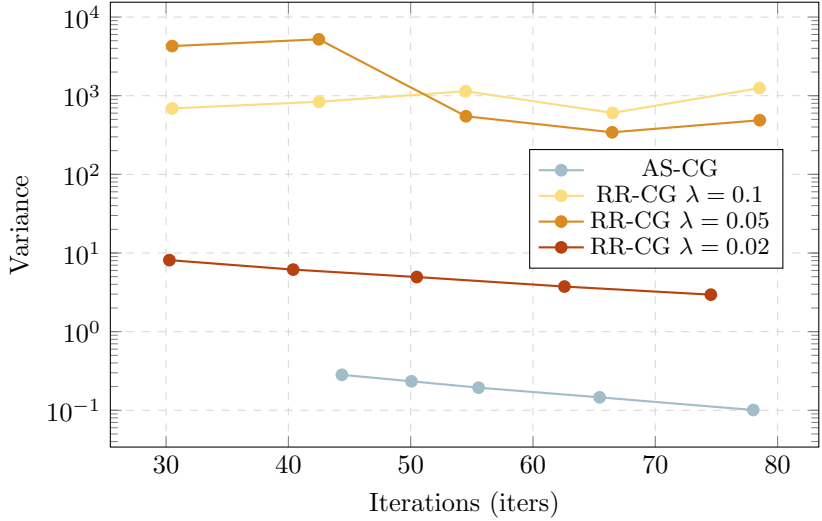
\begin{figure}
    \centering
    \begin{tikzpicture}
    \begin{axis}[
        width=11cm, height=7.5cm, 
        xlabel={Iterations (iters)},
        ylabel={Variance},
        ymode=log, 
        grid=major,
        grid style={dashed, gray!30},
        legend style={at={(0.95,0.67)},anchor=north east, font=\small, row sep=-2pt},
    ]

    \addplot[
        mark=*,
        color=steelblue,
        thick, 
    ] table[x=iters, y=variance, col sep=comma] {csv_for_plot/test_acg_diff_latest.csv};
    \addlegendentry{AS-CG}

    \addplot[
        mark=*,
        color=joshua,
        thick, 
    ] table[x=iters, y=variance, col sep=comma] {csv_for_plot/test_rrcg_diff_latest_01.csv};
    \addlegendentry{RR-CG $\lambda = 0.1$}

    \addplot[
        mark=*,
        color=orange,
        thick, 
    ] table[x=iters, y=variance, col sep=comma] {csv_for_plot/test_rrcg_diff_latest_005.csv};
    \addlegendentry{RR-CG $\lambda = 0.05$}

    \addplot[
        mark=*,
        color=rust,
        thick, 
    ] table[x=iters, y=variance, col sep=comma] {csv_for_plot/test_rrcg_diff_latest_002.csv};
    \addlegendentry{RR-CG $\lambda = 0.02$}

    \end{axis}
\end{tikzpicture}
    \caption[Multiple Linear Systems]{Multiple Linear Systems: the variance with respect to the average number of CG iterations.}
    \label{fig:diff}
\end{figure}

\subsection{GP Hyperparameter Training}
We conduct Gaussian Process (GP) hyperparameter training on the PoleTele dataset from the UCI Machine Learning repository \cite{Asuncion:2007}. We employ the RBF kernel characterized by hyperparameters $\theta$ including outputscale ($\gamma$), lengthscale ($l$), and noise variance ($\sigma^2$):
\begin{equation}
    k(x_i, x_j) = \gamma \cdot \exp\left(-\frac{1}{2l^2} \|x_i - x_j\|^2\right) + \sigma^2 \delta_{i, j}.
\end{equation}
The goal is to find the optimized hyperparameters $\theta^{*}$ by minimizing the loss function
\begin{equation}
\mathcal{L}(\theta) = 
\frac{1}{2}{y}^T{K}^{-1}{y}+\frac{1}{2}\log \det ({K})
+\frac{N}{2}\log 2\pi,
\end{equation}
where $N$ is the number of data points and  ${K}\in\mathbb{R}^{N\times N}$ is the kernel matrix. We calculate the gradient
\begin{equation} \label{eq: gradient}
        \frac{\partial \mathcal{L}}{\partial \theta} = \frac{1}{2}\left(\mathrm{tr}\left({K}^{-1}\frac{\partial {K}}{\partial \theta}\right) - {y}^T {K}^{-1}\frac{\partial {K}}{\partial \theta}{K}^{-1}{y}\right).
\end{equation}
To approximate the trace term $\mathrm{tr}({K}^{-1}\frac{\partial {K}}{\partial \theta})$ in \eqref{eq: gradient}, we utilize stochastic trace estimation \cite{gardner2018gpytorch, cutajar2016preconditioning} to change the problem into the systems of equations:
\begin{equation}
    \mathrm{tr}\left({K}^{-1}\frac{\partial {K}}{\partial \theta}\right) = \mathbb{E}\left[{z_i}^T{K}^{-1}\frac{\partial {K}}{\partial \theta}{z_i}\right] \approx \frac{1}{t}\sum_{i=1}^{t}\left[{z_i}^T{K}^{-1}\frac{\partial {K}}{\partial \theta}{z_i}\right],
\end{equation}
where $z_1, \ldots, z_t$ are random probe vectors with $\mathbb{E}\left[z_i\right] = 0$ and $\mathbb{E}\left[z_i{z_i}^T\right] = I$.

We set $t = 30$ with Rademacher random variables to generate probe vectors. We perform $100$ GP optimization iterations, using the Adam optimizer with learning rate $\alpha = 0.01$ and a scheduler that decreases the learning rate by $70\%$ in $55$ and $90$ iterations.  We adjust the parameters for  AS-CG and RR-CG to target an average of $35$ iterations, consistent with the deterministic CG. For RR-CG , we choose the temperature parameter $\lambda = 0.05$ and $\lambda = 0.10$ as suggested in \cite{potapczynski2021bias}. For both deterministic and stochastic methods, we use the deterministic solution from the previous training step as a warm-start initial guess and run five trials for a fixed set of parameters. We also implement GP optimization using Cholesky factorization as the benchmark.




\begin{figure}
    \centering
    \begin{tikzpicture}
    \begin{groupplot}[
        group style={
            group size=2 by 1,
            horizontal sep=1.8cm,
            vertical sep=1.8cm,
        },
        width=6.4cm, height=6.4cm,
        xlabel={Number of Optimization Steps},
        ylabel={Negative Marginal Log Likelihood},
        grid=major,
        grid style={dashed, gray!30},
        ymin=-0.4, ymax=-0, 
        ytick={0, -0.1, -0.2, -0.3, -0.4}, 
        legend style={
            legend columns=5, 
            at={(1.15, 1.15)}, 
            anchor=south, 
            draw=none, 
            fill=none, 
            font=\large, 
        },
    ]

    \nextgroupplot[
        title={\textbf{AS-CG}}
    ]
    \addplot[
        mark=-,
        color=orange,
        thick,
    ] table[x=iter, y=Loss, col sep=comma, restrict expr to domain={\thisrow{iter}}{55:100000}] {csv_for_plot/new-pol-result/1-acg.csv};
    \addlegendentry{Trial 1}
    \addplot[
        mark=-,
        color=joshua,
        thick,
    ] table[x=iter, y=Loss, col sep=comma, restrict expr to domain={\thisrow{iter}}{55:100000}] {csv_for_plot/new-pol-result/2-acg.csv};
    \addlegendentry{Trial 2}
    \addplot[
        mark=-,
        color=silver,
        thick,
    ] table[x=iter, y=Loss, col sep=comma, restrict expr to domain={\thisrow{iter}}{55:100000}] {csv_for_plot/new-pol-result/3-acg.csv};
    \addlegendentry{Trial 3}
    \addplot[
        mark=-,
        color=steelblue,
        thick,
    ] table[x=iter, y=Loss, col sep=comma, restrict expr to domain={\thisrow{iter}}{55:100000}] {csv_for_plot/new-pol-result/4-acg.csv};
    \addlegendentry{Trial 4}
    \addplot[
        mark=-,
        color=seagreen,
        thick,
    ] table[x=iter, y=Loss, col sep=comma, restrict expr to domain={\thisrow{iter}}{55:100000}] {csv_for_plot/new-pol-result/5-acg.csv};
    \addlegendentry{Trial 5}

    \nextgroupplot[
        title={\textbf{RR-CG $\lambda$ = 0.05}}
    ]
    \addplot[
        mark=-,
        color=orange,
        thick,
        forget plot,
    ] table[x=iter, y=Loss, col sep=comma, restrict expr to domain={\thisrow{iter}}{55:100000}] {csv_for_plot/new-pol-result/1-rrcg_005.csv};
    \addplot[
        mark=-,
        color=joshua,
        thick,
        forget plot,
    ] table[x=iter, y=Loss, col sep=comma, restrict expr to domain={\thisrow{iter}}{55:100000}] {csv_for_plot/new-pol-result/2-rrcg_005.csv};
    \addplot[
        mark=-,
        color=silver,
        thick,
        forget plot,
    ] table[x=iter, y=Loss, col sep=comma, restrict expr to domain={\thisrow{iter}}{55:100000}] {csv_for_plot/new-pol-result/3-rrcg_005.csv};
    \addplot[
        mark=-,
        color=steelblue,
        thick,
        forget plot,
    ] table[x=iter, y=Loss, col sep=comma, restrict expr to domain={\thisrow{iter}}{55:100000}] {csv_for_plot/new-pol-result/4-rrcg_005.csv};
    \addplot[
        mark=-,
        color=seagreen,
        thick,
        forget plot,
    ] table[x=iter, y=Loss, col sep=comma, restrict expr to domain={\thisrow{iter}}{55:100000}] {csv_for_plot/new-pol-result/5-rrcg_005.csv};
    \end{groupplot}
    
\end{tikzpicture} 
    \caption{The GP optimization objective for models trained with AS-CG and RR-CG ($\lambda = 0.05$)}
    \label{fig:GP-random}
\end{figure}
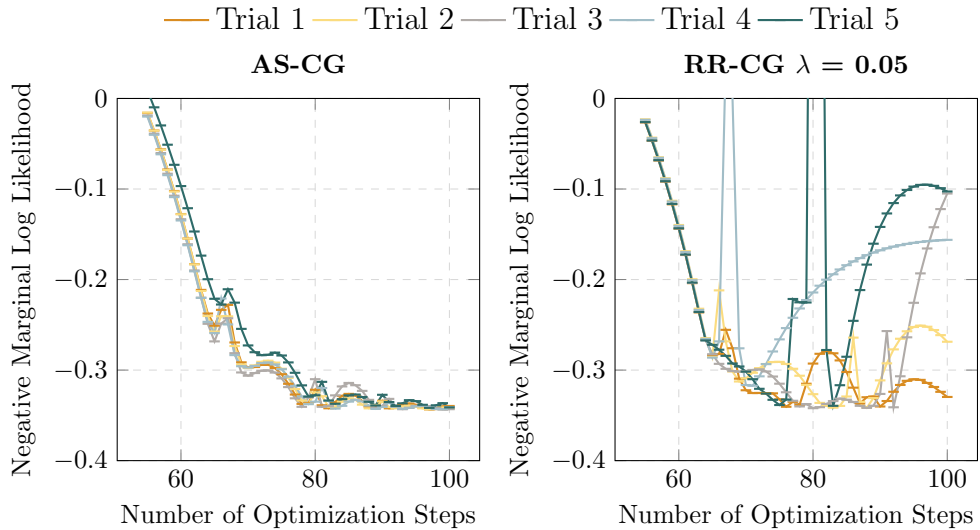

\begin{figure}
    \centering
    \begin{tikzpicture}
    \begin{groupplot}[
        group style={
            group size=2 by 1,
            horizontal sep=1.8cm,
            vertical sep=1.8cm,
        },
        width=6.4cm, height=6.4cm,
        xlabel={Number of Optimization Steps},
        ylabel={Negative Marginal Log Likelihood},
        grid=major,
        grid style={dashed, gray!30},
        legend style={
            legend columns=3, 
            at={(1.17, 1.15)}, 
            anchor=south, 
            draw=none, 
            fill=none, 
            font=\large, 
        },
    ]

    \nextgroupplot[
        ymin=-0.4, ymax=-0.0, 
        ytick={0.0, -0.1, -0.2, -0.3, -0.4}, 
        title={}
    ]
    \addplot[
        mark=-,
        color=steelblue,
        thick,
    ] table[x=iter, y=Loss, col sep=comma, restrict expr to domain={\thisrow{iter}}{55:100000}] {csv_for_plot/new-pol-result/3-acg.csv};
    \addlegendentry{AS-CG}
    \addplot[
        mark=-,
        color=darksky,
        thick,
    ] table[x=iter, y=Loss, col sep=comma, restrict expr to domain={\thisrow{iter}}{55:100000}] {csv_for_plot/new-pol-result/3-cg.csv};
    \addlegendentry{CG}
    \addplot[
        mark=-,
        color=orange,
        thick,
    ] table[x=iter, y=Loss, col sep=comma, restrict expr to domain={\thisrow{iter}}{55:100000}] {csv_for_plot/new-pol-result/chol.csv};
    \addlegendentry{Cholesky}

    \nextgroupplot[
        title={}
    ]
    \addplot[
        mark=-,
        color=steelblue,
        thick,
        forget plot,
    ] table[x=iter, y=Loss, col sep=comma, restrict expr to domain={\thisrow{iter}}{65:100000}] {csv_for_plot/new-pol-result/3-acg.csv};
    \addplot[
        mark=-,
        color=darksky,
        thick,
        forget plot,
    ] table[x=iter, y=Loss, col sep=comma, restrict expr to domain={\thisrow{iter}}{65:100000}] {csv_for_plot/new-pol-result/3-cg.csv};
    \addplot[
        mark=-,
        color=orange,
        thick,
        forget plot,
    ] table[x=iter, y=Loss, col sep=comma, restrict expr to domain={\thisrow{iter}}{65:100000}] {csv_for_plot/new-pol-result/chol.csv};
    \end{groupplot}
    
\end{tikzpicture} 
    \caption{The GP optimization objective for models trained with AS-CG, the deterministic CG, and Cholesky. We pick $1$ trial each from AS-CG and the deterministic CG since they produce very similar training results within their own trials.}
    \label{fig:GP-deter}
\end{figure}
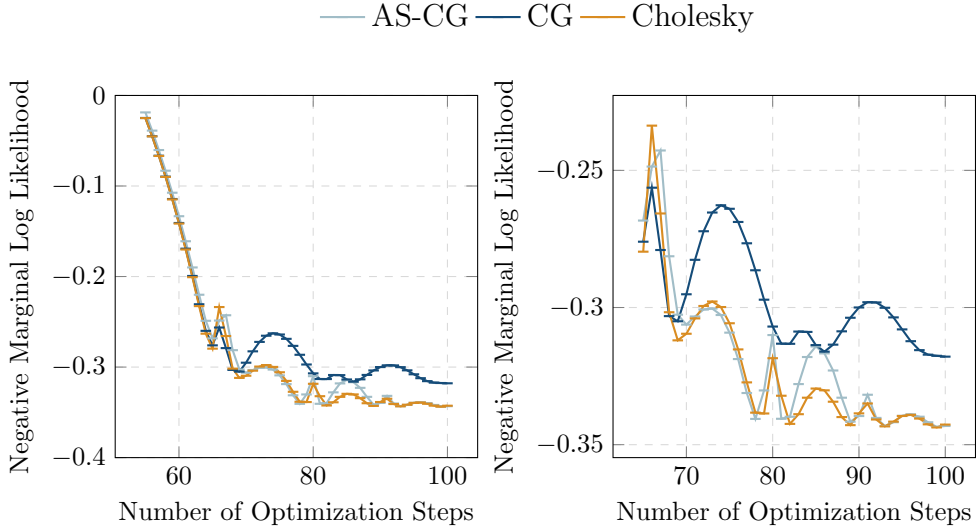

Since the two chosen $\lambda$ values of RR-CG produce very similar training results, we only display the result for $\lambda = 0.05$ for comparison with that of AS-CG in \cref{fig:GP-random}. We notice that the training process of RR-CG has a sharper fluctuation compared with that of AS-CG, which suggests that AS-CG returns the solution with lower variance and thus becomes more stable in GP optimization. RR-CG maintains the same average number of conjugate gradients throughout the training process, while AS-CG reduces computational cost in the early stages and increases the inner loop iterations later to keep the optimization objective stable. However, in terms of the lowest loss, we do not see an obvious difference between AS-CG and RR-CG methods.

We also compare the training performance of AS-CG with deterministic CG and Cholesky methods, as shown in \cref{fig:GP-deter}. Due to the bias introduced by early truncation, the deterministic CG fails to converge to the optimal negative marginal log likelihood. AS-CG and the benchmark Cholesky methods demonstrate comparable training trend and achieve the optimal loss.

\subsection{Competitive Physics Informed Neural Networks} Physics Informed Neural Network (PINNs) is a PDE solver that uses the square of the PDE residual as the loss function. Competitive Physics Informed Neural Networks (CPINNs) \cite{zeng2022competitive} formulates a minimax game between the PINN and a discriminator network. 
It rewards the discriminator and penalizes the PINN if the former correctly predicts mistakes of the latter. Consider a discretized PDE form 
\begin{equation} \label{eq: discrePDE}
    A \pi = f.
\end{equation}
CPINNS turns \eqref{eq: discrePDE} into the minimax problem:
\begin{equation} \label{eq: minimax}
    \min_{\pi} \max_{\delta} \delta^\top\left(A \pi - f\right),
\end{equation}
where $\pi$ and $\delta$ are the weight vectors that the output of the PINN and the discriminator network linearly depend on.
The solution of \eqref{eq: minimax} amounts to solving
\begin{equation} \label{eq: linearsysCPINNs}
\begin{bmatrix} 
    0 & A^\top \\
    A & 0
\end{bmatrix}
\begin{bmatrix}
\pi \\ \delta
\end{bmatrix} =
\begin{bmatrix}
0 \\ f
\end{bmatrix}.
\end{equation}

In this experiment, we train CPINNs by the GMRES-based adaptive competitive gradient descent (GACGD) \cite{cgds-package}, which provides the gradient update with an RMSProp-type heuristic to adjust learning rates. We apply our AS estimator and the exponential decay RR estimator to GMRES, named AS-GACGD and RR-GACGD respectively, and compare their performance with the deterministic algorithm on the Poisson equation and Allen-Cahn equation.

\subsubsection{Poisson Equation}
We first consider a $2$-d Poisson problem
\begin{equation}
    \Delta u(x, y) = -2 \sin(x) \cos(y), \quad x,y \in [-2,2],
    \label{e:poisson}
\end{equation}
with Dirichlet boundary conditions
\begin{equation}
    \begin{aligned}
        u(x,-2)            &= \sin(x)\cos(-2),            \qquad  u(-2,y) = \sin(-2)\cos(y), \\
    u(x,\phantom{-}2)  &= \sin(x)\cos(\phantom{-}2),  \qquad  u(\phantom{-}2,y)  = \sin(\phantom{-}2)\cos(y).
    \end{aligned}
\end{equation}
We adopt the experimental setup described in \cite{zeng2022competitive}. We use $5000$ training points within the domain $[-2, 2]\times[-2,2]$, with $50$ training points on each side of the domain boundary. The PINN model has $3$ hidden layers, and the discriminator model contains $4$ hidden layers. Both models have $50$ neurons in each layer. The GACGD optimizer uses a learning rate of $10^{-3}$, beta values $\beta_1 = 0.99$ and $\beta_2 = 0.99$, and $\epsilon = 10^{-8}$. We adjust the parameters for AS-GACGD and RR-GACGD to maintain an average of $200$ GMRES iterations per training step, consistent with the deterministic GACGD. For all methods, we use the deterministic GMRES solution from the previous training step as a warm-start initial guess for the subsequent step. 

\cref{fig:poisson} shows the average performance of these three algorithms with the uncertainty region. RR-GACGD, due to its high variance, begins to diverge starting from $10^{-6}$. AS-GACGD performs comparably with the deterministic GACGD, which reduces the error to nearly the $10^{-8}$ level.

\begin{figure}
    \centering
    \begin{tikzpicture}
    ]
    \begin{axis}[
        width=10cm, height=6cm, 
        xlabel={Number of GMRES Iterations},
        ylabel={$L_2$ Relative Error},
        ymode=log, 
        xmode=log,
        grid=major,
        grid style={dashed, gray!30},
        legend style={at={(0.05,0.05)},anchor=south west},
    ]
    

    \addplot [
        color=steelblue,
        thick,
        mark=none,
        legend image post style={sharp plot, color=steelblue}
    ] table[x=iter_num_sum, y=L2 Error Mean, col sep=comma, restrict expr to domain={\thisrow{iter_num_sum}}{1000:1000000000}
    ]{csv_for_plot/poisson-result/ascg_result.csv};
    \addlegendentry{AS-GACGD}

    \addplot [
        name path=upper,
        draw=none,
        forget plot
    ] table [
        x=iter_num_sum,
        y expr=\thisrow{L2 Error Mean} + \thisrow{L2 Error Std},
        col sep=comma, restrict expr to domain= {\thisrow{iter_num_sum}}{1000:1000000000}]{csv_for_plot/poisson-result/ascg_result.csv};

    \addplot [
        name path=lower,
        draw=none,
        forget plot
    ] table [
        x=iter_num_sum,
        y expr=\thisrow{L2 Error Mean} - \thisrow{L2 Error Std},
        col sep=comma, restrict expr to domain= {\thisrow{iter_num_sum}}{1000:1000000000}]{csv_for_plot/poisson-result/ascg_result.csv};

    \addplot[
        fill=steelblue,
        opacity=0.2,
        forget plot
    ] fill between [
        of=upper and lower
    ];

    \pgfplotstableread[col sep=comma]{csv_for_plot/poisson-result/rrcg_result.csv}\datatable

    \addplot [
        color=rust,
        thick,
        mark=none,
        legend image post style={sharp plot, color=rust}
    ] table [x=iter_num_sum, y=L2 Error Mean, col sep=comma, restrict expr to domain={\thisrow{iter_num_sum}}{1000:1000000000}
    ]{csv_for_plot/poisson-result/rrcg_result.csv};
    \addlegendentry{RR-GACGD}

    \addplot [
        name path=upper2,
        draw=none,
        forget plot
    ] table [
        x=iter_num_sum,
        y expr=\thisrow{L2 Error Mean} + \thisrow{L2 Error Std}, col sep=comma, restrict expr to domain={\thisrow{iter_num_sum}}{1000:1000000000}
    ]{csv_for_plot/poisson-result/rrcg_result.csv};

    \addplot [
        name path=lower2,
        draw=none,
        forget plot
    ] table [
        x=iter_num_sum,
        y expr=\thisrow{L2 Error Mean} - \thisrow{L2 Error Std}, col sep=comma, restrict expr to domain={\thisrow{iter_num_sum}}{1000:1000000000}
    ]{csv_for_plot/poisson-result/rrcg_result.csv};

    \addplot [
        fill=rust,
        opacity=0.2,
        forget plot
    ] fill between [
        of=upper2 and lower2
    ];



    \addplot [
        color=joshua,
        thick,
        mark=none,
        legend image post style={sharp plot, color=joshua}
    ] table [x=iter_num_sum, y=L2 error, col sep=comma, restrict expr to domain={\thisrow{iter_num_sum}}{1000:1000000000}
    ]{csv_for_plot/poisson-result/cg_result.csv};
    \addlegendentry{GACGD}



    
    \end{axis}
\end{tikzpicture}
    \caption[Poisson Problem]{Comparison of CPINNs with deterministic GMRES and two randomized estimators for the Poisson equations}
    \label{fig:poisson}
\end{figure}
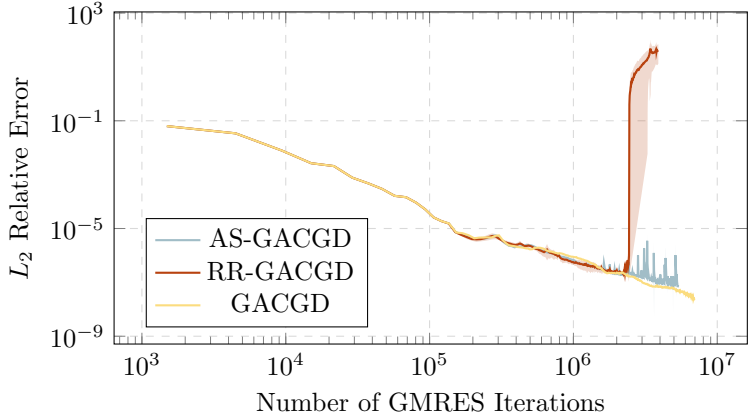

\subsubsection{Allen-Cahn Equation} We next consider an $1$-d Allen-Cahn problem
\begin{equation} 
    u_t - 0.0001u_{xx} +  5 u^3 - 5u = 0, \quad x\in [-1, 1], \quad t\in [0, 1], 
    \label{eq:ac} 
\end{equation}
where $u(t,x)$ is the solution of the PDE and
\begin{equation}
    u(0,x) = x^2\cos(\pi x), \quad u(t, -1) = u(t, 1), \quad u_x(t, -1) = u_x(t, 1)
    \label{eq:acbc}
\end{equation}
are the initial and boundary conditions. 

The training and testing data are from \cite{raissi2019physics}. We adopt the experimental setup described in \cite{zeng2022competitive}: we randomly sample $10000$ training points within the domain for the constraint in \eqref{eq:ac}, with $100$ and $256$ data points for the initial and boundary conditions in \eqref{eq:acbc}. The PINN model has $4$ hidden layers with $128$ neurons in each layer, and the discriminator model contains $4$ hidden layers with $256$ neurons per layer. The GACGD optimizer uses a learning rate of $10^{-3}$, beta values $\beta_1 = 0.99$ and $\beta_2 = 0.99$, and $\epsilon = 10^{-8}$. We adjust the parameters for AS-GACGD and RR-GACGD to ensure that they both have an average of $220$ GMRES iterations per training step, consistent with the deterministic GACGD. Align with the Poisson problem settings, we use the deterministic GMRES solution as a warm-start initial guess. 

\cref{fig:ac} shows the average performance of these three algorithms with their uncertainty regions. Due to the increased difficulty of the Allen-Cahn problem compared to the Possion problem, RR-GACGD starts to diverge around $10^{-1}$. Both AS-GACGD and the deterministic GAGCD achieves lower accuracy of approximately $10^{-2}$. AS-GACGD demonstrates improvement that achieving acurracy below $10^{-3}$ in some trials.

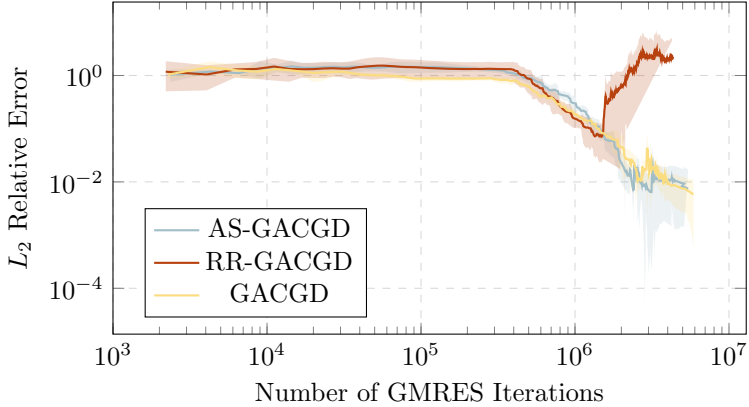
\begin{figure}
    \centering
    \begin{tikzpicture}
    \begin{axis}[
        width=10cm, height=6cm, 
        xlabel={Number of GMRES Iterations},
        ylabel={$L_2$ Relative Error},
        ymode=log, 
        xmode=log,
        grid=major,
        grid style={dashed, gray!30},
        legend style={at={(0.05,0.05)},anchor=south west},
    ]
    
    \addplot [
        color=steelblue,
        thick,
        mark=none,
        legend image post style={sharp plot, color=steelblue}
    ] table [x=iter_num_sum, y=L2 Error Mean, col sep=comma, restrict expr to domain={\thisrow{iter_num_sum}}{1000:1000000000}
    ] {csv_for_plot/ac-result/average_with_uncertainty_17_18_19_ascg.csv};
    \addlegendentry{AS-GACGD}

    \addplot [
        name path=upper,
        draw=none,
        forget plot
    ] table [
        x=iter_num_sum,
        y expr=\thisrow{L2 Error Mean} + \thisrow{L2 Error Std}, col sep=comma, restrict expr to domain={\thisrow{iter_num_sum}}{1000:1000000000}
    ] {csv_for_plot/ac-result/average_with_uncertainty_17_18_19_ascg.csv};

    \addplot [
        name path=lower,
        draw=none,
        forget plot
    ] table [
        x=iter_num_sum,
        y expr=\thisrow{L2 Error Mean} - \thisrow{L2 Error Std}, col sep=comma, restrict expr to domain={\thisrow{iter_num_sum}}{1000:1000000000}
    ] {csv_for_plot/ac-result/average_with_uncertainty_17_18_19_ascg.csv};

    \addplot[
        fill=steelblue,
        opacity=0.2,
        forget plot
    ] fill between [
        of=upper and lower
    ];

    \addplot [
        color=rust,
        thick,
        mark=none,
        legend image post style={sharp plot, color=rust}
    ] table [x=iter_num_sum, y=L2 Error Mean, col sep=comma, restrict expr to domain={\thisrow{iter_num_sum}}{1000:1000000000}
    ] {csv_for_plot/ac-result/average_with_uncertainty_17_18_19_rrcg.csv};
    \addlegendentry{RR-GACGD}

    \addplot [
        name path=upper2,
        draw=none,
        forget plot
    ] table [
        x=iter_num_sum,
        y expr=\thisrow{L2 Error Mean} + \thisrow{L2 Error Std}, col sep=comma, restrict expr to domain={\thisrow{iter_num_sum}}{1000:1000000000}
    ] {csv_for_plot/ac-result/average_with_uncertainty_17_18_19_rrcg.csv};

    \addplot [
        name path=lower2,
        draw=none,
        forget plot
    ] table [
        x=iter_num_sum,
        y expr=\thisrow{L2 Error Mean} - \thisrow{L2 Error Std}, col sep=comma, restrict expr to domain={\thisrow{iter_num_sum}}{1000:1000000000}
    ] {csv_for_plot/ac-result/average_with_uncertainty_17_18_19_rrcg.csv};

    \addplot [
        fill=rust,
        opacity=0.2,
        forget plot
    ] fill between [
        of=upper2 and lower2
    ];

    \addplot [
        color=joshua,
        thick,
        mark=none,
        legend image post style={sharp plot, color=joshua}
    ] table [x=iter_num_sum, y=L2 Error Mean, col sep=comma, restrict expr to domain={\thisrow{iter_num_sum}}{1000:1000000000}
    ] {csv_for_plot/ac-result/average_with_uncertainty_17_18_19_cg.csv};
    \addlegendentry{GACGD}

    \addplot [
        name path=upper3,
        draw=none,
        forget plot
    ] table [
        x=iter_num_sum,
        y expr=\thisrow{L2 Error Mean} + \thisrow{L2 Error Std}, col sep=comma, restrict expr to domain={\thisrow{iter_num_sum}}{1000:1000000000}
    ] {csv_for_plot/ac-result/average_with_uncertainty_17_18_19_cg.csv};

    \addplot [
        name path=lower3,
        draw=none,
        forget plot
    ] table [
        x=iter_num_sum,
        y expr=\thisrow{L2 Error Mean} - \thisrow{L2 Error Std}, col sep=comma, restrict expr to domain={\thisrow{iter_num_sum}}{1000:1000000000}
    ] {csv_for_plot/ac-result/average_with_uncertainty_17_18_19_cg.csv};

    \addplot [
        fill=joshua,
        opacity=0.2,
        forget plot
    ] fill between [
        of=upper3 and lower3
    ];
    
    \end{axis}
\end{tikzpicture}
    \caption[AC Problem]{Comparison of CPINNs with deterministic GMRES and two randomized estimators for the Allen-Cahn equations.}
    \label{fig:ac}
\end{figure}

\section{Conclusions}
We introduce a randomized truncation estimator named the Adaptively Subsampled estimator that fulfills the deficiency of theoretical optimality for CG that satisfies a diminishing returns property. We can surprisingly compute the closed-form truncation probabilities on-the-fly, ensuring its adaptivity for numerical implementation. When the diminishing returns property fails, we prove that obtaining the optimal on-the-fly truncation distribution is impossible, while our generalized AS estimator still provides unbiased, albeit suboptimal, solution. We show how to extend the AS estimator to minimal residual iterative methods for a broader range of application.
 
\section*{Acknowledgments}
We thank Stephen Huan for valuable comments.
This research was supported in part through research cyberinfrastructure resources and services provided by the Partnership for an Advanced Computing Environment (PACE) at the Georgia Institute of Technology, Atlanta, Georgia, USA.
Both authors gratefully acknowledge support from the Office of Naval Research under award number N00014-23-1-2545 (Untangling Computation).

\bibliographystyle{siamplain}
\bibliography{references}
\end{document}